\documentclass[reqno, 11pt]{amsart}

\usepackage{csquotes}  
\usepackage{amsmath, amssymb}
\usepackage[margin=0.85in]{geometry}

\usepackage{amsfonts}

\renewcommand{\theequation}{\thesection.\arabic{equation}}

\newtheorem{thm}{Theorem}[section]
 \newtheorem{lem}{Lemma}[section]
 \newtheorem{prop}{Proposition}[section]
 \newtheorem{cor}{Corollary}[section]

\newtheorem{rem}{Remark}[section]
\renewcommand{\v}[1]{\ensuremath{\mathbf{#1}}}

\title{Gradient bounds for a thin film epitaxy equation}

\author{Dong Li}
\address{Department of Mathematics, University of British Columbia, 1984 Mathematics Road,
Vancouver, BC, Canada V6T 1Z2}
\email{dli@math.ubc.ca}
\author{Zhonghua Qiao}
\address{Corresponding author. Department of Applied Mathematics, The Hong Kong Polytechnic University, Hung Hom, Hong Kong}
\email{zhonghua.qiao@polyu.edu.hk}
\author{Tao Tang}
\address{
Department of Mathematics, South University of Science and Technology, Shenzhen, Guangdong
518055, China, and Department of Mathematics, Hong Kong Baptist University, Kowloon, Hong Kong}
\email{tangt@sustc.edu.cn}

\begin{document}
\maketitle







\renewcommand{\theequation}{\thesection.\arabic{equation}}

\renewcommand{\v}[1]{\ensuremath{\mathbf{#1}}}


\begin{abstract}
We consider a gradient flow modeling the epitaxial growth of thin films
with slope selection. The surface height profile satisfies a nonlinear diffusion
equation with biharmonic dissipation.  We establish optimal local and
global wellposedness for initial data with critical regularity. To understand the mechanism
of slope selection and the dependence on the dissipation coefficient,
we exhibit several lower and upper bounds for the gradient of the solution
in physical dimensions $d\le 3$.
\end{abstract}


\section{Introduction}
Let $\nu>0$. Consider
\begin{align}\label{0}
\partial_t h = \nabla \cdot ( (|\nabla h|^2-1)\nabla h ) - \nu \Delta^2 h
\end{align}
and the 1D version
\begin{align}  \label{2}
h_t = (h_x^3-h_x)_{x}- \nu h_{xxxx}.
\end{align}
Eq. \eqref{0} is a nonlinear diffusion equation which models the epitaxial growth
of thin films. It is posed on the spatial domain $\Omega$ which can either be the whole space
$\mathbb R^d$, the $L$-periodic torus ($L>0$ is a parameter corresponding to the size of the system)
 $\mathbb R^d/L \mathbb Z^d$,
or a finite domain in $\mathbb R^d$ with suitable boundary conditions. In this work for simplicity
we shall be mainly
concerned with the $2\pi$-periodic case $\Omega=\mathbb T^d=\mathbb R^d/2\pi \mathbb Z^d$
but our results can be easily generalized to other settings.
The function $h=h(t,x): \mathbb R \times \Omega \to \mathbb R$ represents the scaled
height of a thin film and $\nu>0$ is positive parameter which is sometimes called the diffusion
coefficient. Typically in numerical simulations one is interested in the regime where $\nu$ is small
so that the nonlinear effects become dominant. The 1D version \eqref{2} is connected with the
Cahn-Hilliard equation:
\begin{align*}
\partial_t u = \Delta ( u^3-u) -\nu \Delta^2 u
\end{align*}
through the identification $u=\partial_x h$. This connection breaks down for dimension $d\ge 2$.

Define the energy
\begin{align} \label{0_Eh}
E(h) = \int_{\Omega} \Bigl( \frac 14 ( |\nabla h|^2-1)^2 + \frac {\nu} 2  |\Delta h|^2 \Bigr) dx.
\end{align}
The equation \eqref{0} can be regarded as a gradient flow of the energy functional $E(h)$  in
 $L^2(\Omega)$. In fact, it is easy to check that
 \begin{align} \label{0_Eh_1}
 \frac d {dt} E(h) = -\|\partial_t h\|_2^2,
 \end{align}
i.e. the energy is always decreasing in time as far as smooth solutions are concerned.
{Alternatively one can derive  the energy law
from \eqref{0} by multiplying both sides by $\partial_t h$ and integrating by parts.}
The first term in \eqref{0_Eh} models the Ehrlich-Schowoebel effect \cite{EH66,Sc66, Sc69}.
Formally speaking it forces the slope of the thin film $|\nabla h| \approx 1$. For this
reason Eq. \eqref{0} is often called the growth equation with slope selection.
On the other hand, in the literature there are also
models \enquote{\,without slope selection\,}, such as
\begin{align} \label{0_without}
\partial_t h = - \nabla \cdot \bigl( \frac 1 {1+|\nabla h|^2} \nabla h \bigr) - \nu \Delta^2 h.
\end{align}
Heuristically speaking, if in \eqref{0_without} the slope $|\nabla h|$ is small, then
\begin{align*}
\frac 1 {1+|\nabla h|^2} \approx 1- |\nabla h|^2
\end{align*}
and one recovers the nonlinearity in \eqref{0}. However this line of argument seems only reasonable
when $|\nabla h|\ll 1$ which is a typical transient regime and  not so appealing physically.
Indeed the long time interfacial dynamics governed by \eqref{0} and \eqref{0_without} can be quite different,
see for example the discussion in \cite{LL03}.  The second term in \eqref{0_Eh} corresponds
to the fourth-order diffusion in \eqref{0}. It has a stabilizing effect both theoretically and numerically.

Eq. \eqref{0} can also be viewed as regularized version of the
equation
\begin{align}
\partial_t h = \nabla \cdot ( (|\nabla h|^2-1) \nabla h).  \label{00}
\end{align}
The wellposedness of \eqref{00} is a rather subtle issue. {In light of recent developments
(\cite{BL13,BL14}), one should expect
generic illposedness although the underlying mechanism will be different.}
 However as
it turns out, if there is a smooth solution to \eqref{00} on some
finite time interval, then it must admit some form of a maximum
principle. We record it here as

\begin{prop} \label{prop_max_0}
[Maximum principle for smooth solutions to \eqref{00}] Let the
dimension $d\ge 1$ and $\mathbb T^d = \mathbb R^d/2\pi \mathbb Z^d$
be the usual $2\pi$-periodic torus. Let $T>0$ and assume $h\in C_t^1
C_x^2([0,T]\times \mathbb T^d) $ is a classical solution to
\eqref{00}. Then
\begin{align}
\| \nabla h(t, \cdot) \|_{\infty} \le \max \{  \| \nabla
h(0,\cdot) \|_{\infty},\, 1 \}, \qquad \forall\, 0 \le t \le T.
\label{prop00_e0}
\end{align}
If the dimension $d=1$, then a better bound is available:
\begin{align}
\| \partial_x h(t, \cdot) \|_{\infty} \le \max \{  \| \partial_x
h(0,\cdot) \|_{\infty},\, \frac 1 {\sqrt 3} \}, \qquad \forall\, 0
\le t \le T. \label{prop00_e0a}
\end{align}
\end{prop}

We stress that Proposition \ref{prop_max_0} is a \emph{conditional} result,
namely it assumes the existence of a smooth solution. On the other hand the wellposedness of
classical solutions to the regularized
equation \eqref{0} is much easier to obtain thanks to the fourth order dissipation on the right
hand side. In the Fourier space, the biharmonic operator $-\Delta^2$ seems to offer
much stronger dissipation and damping effect than the usual Laplacian operator, as can be seen
from studying the linear equations
\begin{align*}
\partial_t h = Ah, \quad A=\Delta\; \;\text{or} \; -\Delta^2.
\end{align*}
Since equation \eqref{0} can be viewed as a regularized version of
\eqref{00}, it is very natural to stipulate that solutions to
\eqref{0} should behave much better than those to \eqref{00} from a
general perspective. From this heuristics, it is very tempting to
expect that Proposition \ref{prop_max_0} also holds for \eqref{0}.
Preliminary numerical experiments seem to support this, thus

\textbf{Conjecture 1:} Let $\nu>0$. For general smooth initial data $h_0$, the corresponding solution
 $h=h(t,x)$ to \eqref{0}
satisfies the bound
\begin{align*}
\| \nabla h(t) \|_{\infty} \le \max \{ \| \nabla h_0\|_{\infty}, 1 \}, \qquad \forall\, t>0.
\end{align*}

A weaker form of Conjecture 1 is the following:

\textbf{Conjecture 2:} Let $\nu>0$. For general smooth initial data $h_0$, the corresponding solution
 $h=h(t,x)$ to \eqref{0}
satisfies the bound
\begin{align*}
\| \nabla h(t) \|_{\infty} \le \max \{ \| \nabla h_0\|_{\infty}, \alpha_d\}, \qquad \forall\, t>0.
\end{align*}
where $\alpha_d>0$ is a constant depending only on the dimension $d$.

{Perhaps a better formulation of Conjecture
2 is that $\|\nabla h(t) \|_{\infty} \le F(\|\nabla h_0\|_{\infty}, d)$ for some function $F$
independent of $(\nu,d)$.}
The main point in both Conjecture 1 and
Conjecture 2 is that the constants in the upper bounds of $\|\nabla h\|_{\infty}$
 are \emph{independent
of $\nu$}. If true these gradient bounds can lead to better stability estimates of
numerical algorithms (see \cite{XT06, QZT11, WWW10, SWWW12, LQT14b, LQ16a, LQ16b}).

On the other hand, it is not so difficult
to extract a $\nu$-dependent upper bound on $\| \nabla h \|_{\infty}$, see Corollary \ref{cor_grad} below.

Perhaps a bit surprisingly, the goal of this paper is to disprove Conjecture 1.
Conjecture 2 is still open at the time of this writing.
However we shall give a lower bound for the constant in
Conjecture 2. Namely, we shall show that $\alpha_d\ge C_d>1$ for some explicit constant $C_d$ depending on
the dimension $d$.


To make the paper self-contained, we first establish local and global wellposedness for \eqref{0}.
For $H^2$ initial data in dimensions $d=1,2,3$, a fairly satisfactory
wellposedness theory has been worked out in \cite{LL03} using energy estimates and Galerkin approximation.
By using the method of mild solutions, our Theorem \ref{thm_lwp} below slightly refines
this wellposedness result and allows initial data to be in
the \enquote{critical} space $H^{\frac d2}$ which in particular contains $H^2$ for $d\le 3$. Note that although \eqref{0}
is \emph{not} scale-invariant, in high frequency approximation, one can regard \eqref{0} as
\begin{align} \label{0_hgh}
\partial_t h = \nabla \cdot ( |\nabla h|^2 \nabla h) - \nu \Delta^2 h.
\end{align}
To invoke scaling analysis, one can consider \eqref{0_hgh} posed on the whole space $\mathbb R^d$. If $h(t,x)$
is a solution to \eqref{0_hgh}, then for any $\lambda>0$,
\begin{align} \notag
h_{\lambda}(t,x) =   h(\lambda^4 t, \lambda x)
\end{align}
is also a solution. From this one can deduce that the critical space for \eqref{0_hgh} is $L^{\infty}_x(\mathbb R^d)$
or $\dot H_x^{\frac d2}(\mathbb R^d)$. Thus we have

\begin{thm}[Improved local wellposedness] \label{thm_lwp}
Let the dimension $d\ge 1$. Consider \eqref{0} on the
$2\pi$-periodic torus $\mathbb T^d$ with $\nu>0$. Let $s_d=d/2$. For
any initial data $h_0 \in H^{s_d}(\mathbb T^d)$, there exist $T_0=
T(h_0)>0$ and a unique local solution $h \in C_t^0 H^{s_d}_x$ with
$t^{\frac 14} \nabla h \in C_t^0 C_x^0$, $t^{\frac 14} h \in C_t^0
H^{s_d+1}_x$. Moreover $h(t) \in H_x^m$ for all $m\ge 1$, $0<t<T_*$, where
$0<T_*\le \infty$ is the maximal lifespan of the local solution. In
particular $h(t) \in C_x^{\infty}$ for all $0<t<T_*$. If $h_0$ has mean
zero, then $h(t)$ also has mean zero for all $0<t<T_*$.
\end{thm}

As is well-known, the long time dynamics is dictated by conserved quantities (or conservation laws).
For \eqref{0}, the energy dissipation law \eqref{0_Eh_1} gives a priori $H^2$ control of
the solution with mean zero. {Note that if $h$ has mean zero, then $\|h\|_2$ is controlled
by $\|\Delta h\|_2$ thanks to the Poincar\'e inequality. Or one can just prove it directly
using the Fourier series.} The space $H^2$ is subcritical in dimensions
$d\le 3$ since the corresponding critical space is $H^{\frac d2}$. Thus

\begin{cor}[Global wellposedness for $d\le 3$] \label{cor_gwp}
Let the dimension $d=1,2,3$.  Consider \eqref{0} on the
$2\pi$-periodic torus $\mathbb T^d$ with $\nu>0$. For any initial
data $h_0 \in H^{\frac d2}(\mathbb T^d)$ with mean zero, the corresponding solution $h=h(t,x)$ to
\eqref{0} obtained in Theorem \ref{thm_lwp} exists globally in time.
\end{cor}

\begin{rem}
An interesting open problem is to show the global wellposedness of \eqref{0} in dimension $d=4$.
In that case $H^2$ is the critical space.
\end{rem}

The following Corollary gives gradient bounds on $ h$. For simplicity we assume the initial
data $h_0 \in H^2(\mathbb T^d)$ so that the energy is well-defined. By using the smoothing effect
one can also treat the case
$h_0 \in H^{\frac d2}(\mathbb T^d)$ with the help of Theorem \ref{thm_lwp}. However the bounds in
that case have slightly worse dependence on $\nu$
(for initial transient time when the smoothing effect takes place). We shall not dwell on this subtle
issue here and focus instead on the long time bounds. In Corollary \ref{cor_grad} below, we shall only
consider the case when the diffusion coefficient $\nu$ is not so large (the physically relevant
case is $\nu\to 0$), which we denote by the notation
$0<\nu\lesssim 1$. It means $0<\nu \le \nu_0$ where $\nu_0>0$ is some constant of order $1$. The numerical
value of $\nu_0$ is not so important. For example one can just take $\nu_0=1$.

\begin{cor}[Gradient bounds for $d\le 3$] \label{cor_grad}
Let the dimension $d=1,2,3$. Consider \eqref{0} on the
$2\pi$-periodic torus $\mathbb T^d$ with $0<\nu\lesssim 1$.
Assume $h_0 \in H^2(\mathbb T^d)$ with mean zero. Let $h=h(t,x)$ be
the corresponding global solution to \eqref{0}.
Denote
\begin{align*}
E_0= \int_{\mathbb T^d} \Bigl( \frac 12 \nu |\Delta h_0|^2 + \frac 14 (|\nabla h_0|^2-1)^2 \Bigr) dx.
\end{align*}
Then $\nabla h$ admits the following bounds: for  some absolute constants
$C_1$, $C_2$, $C_3>0$,

\begin{align*}
&\sup_{0\le t<\infty}\| \nabla h(t) \|_{\infty} \le C_1 \nu^{-\frac 16} E_0^{\frac 16} (E_0^{\frac 16}+1),
\qquad \text{if $d=1$}; \\
& \sup_{1\lesssim t<\infty} \| \nabla h(t) \|_{\infty} \le C_2 ( \frac{E_0}{\nu})^{\frac 12} |\log( \frac{E_0+1}{\nu})|,
\qquad \text{if $d=2$}; \\
& \sup_{1\lesssim t<\infty}\| \nabla h(t) \|_{\infty} \le C_3 \nu^{-\frac 32} (E_0+1)^{\frac 32},
\qquad \text{if $d=3$}.
\end{align*}

Similarly for some absolute constants $C_2^{\prime}>0$, $C_3^{\prime}>0$,
\begin{align*}
& \sup_{0\le t\lesssim 1} \| \nabla h(t)- \nabla e^{-\nu t \Delta^2} h_0 \|_{\infty} \le C_2^{\prime}
\cdot( \frac{E_0}{\nu})^{\frac 12} |\log( \frac{E_0+1}{\nu})|,
\qquad \text{if $d=2$}; \\
& \sup_{0\le t\lesssim 1}\| \nabla h(t) - \nabla e^{-\nu t \Delta^2} h_0\|_{\infty} \le C_3^{\prime}
\nu^{-\frac 32} (E_0+1)^{\frac 32},
\qquad \text{if $d=3$}.
\end{align*}

\end{cor}

\begin{rem}
The above gradient bound for $d=1$ follows trivially from energy law and interpolation inequalities. It does
not use the dynamics at all. On the other hand the proof of the bounds for $d=2,3$ uses the mild formulation
of the equation together with energy law. In terms of the dependence on $\nu$
 the bounds here seem not optimal. See for example Proposition \ref{prop_refine_1D}--\ref{prop_refine_1D_long}
 in \S
 \ref{sec_refined_1} for more refined results.
\end{rem}

To disprove Conjecture 1, we shall use two different methods. The first method (see Theorem \ref{thm0} and
Corollary \ref{cor0} below) gives a weak lower bound approximately of the form $1+O(\nu)$ (with $O(\nu)>0$).
Even though this
already settles Conjecture 1 in the negative, the obtained lower bound approaches
to $1$ as $\nu$ tend to zero which is the drawback of the construction.
 On the other hand, the second method (see Theorem \ref{thm1}) gives a $\nu$-independent lower bound which also
 yields a lower bound for the constant $\alpha_d$ in Conjecture 2. It is quite possible that these bounds can
 be improved further.

We now introduce the first construction. To elucidate the main idea, we first state the 1D version.

\begin{thm} \label{thm0}
Consider \eqref{2} with $\nu>0$ and $2\pi$-periodic boundary condition. There exists a family
$\mathcal A$ of smooth initial data such that the following holds:

\begin{enumerate}
\item For any $h_0 \in \mathcal A$, we have  $\int_{\mathbb T} h_0(x) dx=0$ and $\|\partial_x h_0 \|_{\infty} <1$.
\item For any $h_0 \in \mathcal A$, there exists $t_0>0$ (depending on $h_0$) such that the corresponding solution to
\eqref{2} satisfies
\begin{align*}
\| \partial_x h(t_0,\cdot) \|_{\infty} >1.
\end{align*}
\end{enumerate}
\end{thm}

It is relatively straightforward to generalize the construction in Theorem \ref{thm0} to the equation
\eqref{0} in all dimensions.

\begin{cor} \label{cor0}
Let the dimension $d\ge 1$ and $\mathbb T^d$ be the usual $2\pi$-periodic torus.
Consider \eqref{0} with $\nu>0$ and on $(t,x) \in [0,\infty)\times \mathbb T^d$. There exists a family
$\mathcal A$ of smooth initial data such that the following holds:

\begin{enumerate}
\item For any $h_0 \in \mathcal A$, we have $\int_{\mathbb T^d} h_0(x) dx =0$ and $\|\partial_x h_0 \|_{\infty} <1$.
\item For any $h_0 \in \mathcal A$, there exists $t_0>0$ (depending on $h_0$) such that the corresponding solution to
\eqref{0} satisfies
\begin{align*}
\| \nabla h(t_0,\cdot) \|_{\infty} >1.
\end{align*}
\end{enumerate}

\end{cor}

We now introduce the second construction. The key idea builds on examining the linear evolution $e^{-\nu t \Delta^2}$,
and treating the nonlinear part as a correction.

\begin{thm} \label{thm1}
Let the dimension $d\ge 1$ and $\mathbb T^d$ be the usual
$2\pi$-periodic torus. Consider \eqref{0} with $\nu>0$ and on $(t,x)
\in [0,\infty)\times \mathbb T^d$. There exists a constant $C_d>1$
depending only on the dimension $d$, such that for any $\epsilon>0$,
there exists $h_0 \in C^{\infty}(\mathbb T^d)$ for which the
following hold:
\begin{enumerate}
\item $\int_{\mathbb T^d} h_0(x)dx=0$ and $\|\nabla h_0 \|_{\infty} <1$.
\item There exists $t_0>0$ such that the corresponding solution to
\eqref{2} satisfies
\begin{align*}
\| \nabla h(t_0,\cdot) \|_{\infty} >C_d-\epsilon.
\end{align*}
\end{enumerate}

\end{thm}

\begin{rem}
Let
$\displaystyle
f(x) = \frac 1 {(2\pi)^d} \int_{\mathbb R^d} e^{-|\xi|^4} e^{i \xi
\cdot x} d\xi.
$
The constant $C_d$ in Theorem \ref{thm1} is given by
$\displaystyle
C_d= \| f\|_{L_x^1(\mathbb R^d)}>1.
$

\end{rem}

\begin{rem}
One can also consider the following version of \eqref{0} with fractional dissipation:
\begin{align} \label{0_fract}
\partial_t h = \nabla \cdot ( (|\nabla h|^2-1) \nabla h ) - \nu |\nabla|^{\gamma} h,
\end{align}
where $\gamma>2$ controls the \enquote{order} of dissipation. For $h:\, \mathbb T^d \to \mathbb R$,
$|\nabla|^{\gamma}$ can be defined on the Fourier side as
\begin{align*}
\widehat{|\nabla|^{\gamma} h}(k) = |k|^{\gamma} \hat h(k), \qquad \, k \in \mathbb Z^d.
\end{align*}
The $L^{\infty}$-maximum principle holds for the fractional heat propagator $e^{-t |\nabla|^{\gamma}}$ for
$0\le \gamma\le 2$. {The behavior of $e^{-t|\nabla|^{\gamma}}$ for $\gamma<2$ and
the heat operator $e^{t\Delta}$ can be quite different, see for example \cite{L13} for
a discussion in the (Littlewood-Paley) frequency-localized context.}
In the wider setting one can even consider operators of the form $\mathcal
A=|\nabla|^{\gamma}/\log^{\beta}(\lambda+|\nabla|)$ (for $0\le \gamma\le 2$, $\beta\ge 0$ and $\lambda>1$)
and establish a new generalized maximum principle (see \cite{DL14}) for the drift equation
\begin{align*}
\partial_t \theta + v\cdot \nabla \theta = -\mathcal A \theta,
\end{align*}
where $v$ is a given arbitrary external velocity field transporting the scalar quantity $\theta$.
On the other hand, in the regime $\gamma>2$, the $L^{\infty}$-maximum principle is no longer expected since
the corresponding fundamental solution may change signs. Based on this, an analogue of Theorem \ref{thm1} is
expected to hold for
\eqref{0_fract} when $\gamma>2$. In that case the constant $C_d$ is replaced by
\begin{align*}
C_{d,\gamma}= \| \mathcal F^{-1} (e^{-|\xi|^{\gamma}} ) \|_{L_x^1(\mathbb R^d)} >1.
\end{align*}

\end{rem}

\section{Notation and preliminaries}
In this section we collect  some notation and preliminaries used in this paper.

 For any $x=(x_1,\cdots,x_d)\in \mathbb R^d$, we use the Japanese bracket notation
$\langle x \rangle =
\sqrt{1+x_1^2+\cdots+x_d^2}$.

 We denote by $\mathbb T^d= \mathbb R^d / 2\pi \mathbb Z^d$ the $2\pi$-periodic torus.

 Let $\Omega=\mathbb R^d$ or $\mathbb T^d$, $d\ge 1$.
For any function $f:\; \Omega\to
\mathbb R$, we use $\|f\|_{L^p}=\|f\|_{L^p(\Omega)}$ or sometimes $\|f\|_p$ to denote
the  usual Lebesgue $L^p$ norm  for $1 \le p \le
\infty$. If $f=f(x,y): \, \Omega_1\times\Omega_2 \to \mathbb R$, we shall denote by
$\|f\|_{L_x^{p_1} L_y^{p_2}}$ to denote the mixed-norm:
\begin{align*}
\|f\|_{L_x^{p_1} L_y^{p_2}} = \Bigl\|  \|f(x, y)\|_{L_y^{p_2}(\Omega_2)} \Bigr\|_{L_x^{p_1}(\Omega_1)}.
\end{align*}
In a similar way one can define other mixed-norms such as $\|f \|_{C_t^0 H_x^m}$ etc.

   For any two quantities $X$ and $Y$, we denote $X \lesssim Y$ if
$X \le C Y$ for some constant $C>0$. Similarly $X \gtrsim Y$ if $X
\ge CY$ for some $C>0$. We denote $X \sim Y$ if $X\lesssim Y$ and $Y
\lesssim X$. The dependence of the constant $C$ on
other parameters or constants are usually clear from the context and
we will often suppress  this dependence. We denote $X \lesssim_{Z_1,\cdots,Z_m} Y$ if
$X\le C Y$ where the constant $C$ depends on the parameters $Z_1,\cdots,Z_m$.

 We adopt the following convention for Fourier transform pair on $\mathbb R^d$:
\begin{align*}
 &(\mathcal F f)(\xi)=\hat f (\xi) = \int_{\mathbb R^d} f(x) e^{-i x\cdot \xi} dx,  \\
 &f(x) = \frac 1 {(2\pi)^d} \int_{\mathbb R^d} \hat f(\xi) e^{i x \cdot \xi} d\xi.
\end{align*}
Sometimes the inverse Fourier transform is denoted as $\mathcal F^{-1}$.

Also for $f:\, \mathbb T^d \to \mathbb R$, and $k\in \mathbb Z^d$, we denote
the Fourier coefficient
\begin{align*}
\hat f (k ) =\int_{\mathbb T^d} f(x) e^{-i k\cdot x} dx.
\end{align*}
Of course (under suitable conditions) $f$ can be recovered from the Fourier series:
\begin{align*}
f(x) = \frac 1 {(2\pi)^d} \sum_{k \in \mathbb Z^d} \hat f(k) e^{i k\cdot x}.
\end{align*}
Note that if we regard $f$ as a periodic function on $\mathbb R^d$, then
\begin{align} \label{per_delta}
(\mathcal F f )(\xi) = \sum_{k \in \mathbb Z^d} \hat f(k) \delta(\xi-k),
\end{align}
where $\delta$ is the usual Dirac delta distribution on $\mathbb R^d$.

 For $f:\, \mathbb T^d\to \mathbb R$ and $s\ge 0$, we define the $H^s$-norm and $\dot H^s$-norm of $f$ as
\begin{align*}
\| f \|_{H^s}= \Bigl( \sum_{k\in \mathbb Z^d} (1+|k|^{2s}) |\hat f(k) |^2 \Bigr)^{\frac 12}, \quad
\| f \|_{\dot H^s}= \Bigl( \sum_{k\in \mathbb Z^d} |k|^{2s} |\hat f(k) |^2 \Bigr)^{\frac 12}.
\end{align*}
provided of course the above sums are finite. If $f$ has mean zero, then $\hat f(0) =0$ and in this
case
\begin{align*}
\|f \|_{H^s} \sim
\Bigl( \sum_{k\in \mathbb Z^d} |k|^{2s} |\hat f(k) |^2 \Bigr)^{\frac 12}.
\end{align*}

Occasionally we will need to use the Littlewood--Paley (LP) frequency projection
operators. To fix the notation, let $\phi_0 \in
C_c^\infty(\mathbb{R}^d )$ and satisfy
\begin{equation}\nonumber
0 \leq \phi_0 \leq 1,\quad \phi_0(\xi) = 1\ {\text{ for}}\ |\xi| \leq
1,\quad \phi_0(\xi) = 0\ {\text{ for}}\ |\xi| \geq 2.
\end{equation}
Let $\phi(\xi):= \phi_0(\xi) - \phi_0(2\xi)$ which is supported in $1/2\le |\xi| \le 2$.
For any $f \in \mathcal S^{\prime}(\mathbb R^d)$, $j \in \mathbb Z$, define
\begin{align*}
 &\widehat{\Delta_j f} (\xi) = \phi(2^{-j} \xi) \hat f(\xi), \\
 &\widehat{S_j f} (\xi) = \phi_0(2^{-j} \xi) \hat f(\xi), \qquad \xi \in \mathbb R^d.
\end{align*}

 We recall the Bernstein estimates/inequalities: for $1\le p\le q\le \infty$,
\begin{align*}
&\| |\nabla|^s \Delta_j f \|_{L^p(\mathbb R^d)} \sim 2^{js}  \| \Delta_j f \|_{L^p(\mathbb R^d)}, \qquad s \in \mathbb R; \\
& \|  S_j f \|_{L^q(\mathbb R^d)}  +\| \Delta_j f \|_{L^q(\mathbb R^d)}
 \lesssim 2^{j d( \frac 1p - \frac 1 q)} \| f \|_{L^p(\mathbb R^d)}.
\end{align*}

We also need the Bernstein inequalities for periodic functions. Let $f:\, \mathbb T^d \to \mathbb R$ be
a smooth function and \enquote{lift}
$f$ to be a periodic function on $\mathbb R^d$. Then in this way
 $f\in \mathcal S^{\prime}(\mathbb R^d)$ and one can
define $\Delta_j f$ for any $j\in \mathbb Z$. By expressing $\Delta_j f$ in terms of a convolution integral, it is
easy to check
 that $\Delta_j f$ is also a periodic function on $\mathbb R^d$ and thus can be identified as a function
on $\mathbb T^d$. {A more \enquote{direct} way is just to use \eqref{per_delta} and recognize
 $\Delta_j f$ as (on the Fourier side) the partial sum of $\delta$-distributions in a dyadic block. }
 It is then natural to expect that the following ``Bernstein"-type  inequalities hold (note that the norms
are evaluated on $\mathbb T^d$): for any $1\le p\le q\le \infty$,
\begin{align}
&\| |\nabla|^s \Delta_j f \|_{L^p(\mathbb T^d)} \sim 2^{js}  \| \Delta_j f \|_{L^p(\mathbb T^d)}, \qquad s \in \mathbb R;
\label{b_per_e1}\\
& \|\Delta_j f \|_{L^q(\mathbb T^d)} \lesssim 2^{jd(\frac 1p-\frac 1q)} \|f\|_{L^p(\mathbb T^d)},
\qquad\, j \in \mathbb Z;
\label{b_per_e2}\\
& \|  S_j f \|_{L^q(\mathbb T^d)}
 \lesssim 2^{j d( \frac 1p - \frac 1 q)} \| f \|_{L^p(\mathbb T^d)}, \qquad j\ge -2. \label{b_per_e3}
\end{align}
If $f$ has mean zero (so that $\hat f(0)=0$), then one does not need the condition $j\ge -2$
(since $S_{j} f=0$ for $j<-2$). Although these inequalities are standard, we include the proof here
for the sake of completeness.

\begin{proof}[Proof of \eqref{b_per_e1}--\eqref{b_per_e3}]
We shall only prove \eqref{b_per_e1}--\eqref{b_per_e2}.
The proof of \eqref{b_per_e3} is similar to \eqref{b_per_e2}.

First we deal with \eqref{b_per_e1}. For some Schwartz function $\psi$
($\psi=\mathcal F^{-1}(|\xi|^s \phi(\xi))$), we have
\begin{align*}
(|\nabla|^s \Delta_j f)(x) & = 2^{js} \int_{\mathbb R^d} 2^{jd} \psi(2^j(x-y)) f(y) dy \notag \\
& = 2^{js} \sum_{k \in \mathbb Z^d} \int_{\mathbb T^d} 2^{jd} \psi( 2^j (x-y +2\pi k) ) f(y) dy \notag \\
& = 2^{js} \int_{\mathbb T^d} \tilde \psi_j(x-y) f(y)dy,
\end{align*}
where
$\displaystyle
\tilde \psi_j(z) = \sum_{k\in \mathbb Z^d} 2^{jd} \psi(2^j (z+ 2\pi k))
$
is a periodic function on $\mathbb R^d$ (and thus can be identified as a function on $\mathbb T^d$).
By using Young's inequality on $\mathbb T^d$, we get
\begin{align*}
\| |\nabla|^s \Delta_j f \|_{L^p(\mathbb T^d)} \lesssim 2^{js} \| \tilde \psi_j \|_{L^1(\mathbb T^d)}
\| f\|_{L^p(\mathbb T^d)}.
\end{align*}
Easy to check that
$$ \|\tilde \psi_j \|_{L^1(\mathbb T^d)} \le 2^{jd} \| \psi (2^j z) \|_{L_z^1(\mathbb R^d)}
= \| \psi\|_{L^1(\mathbb R^d)} \lesssim 1.$$
Therefore
\begin{align*}
\| |\nabla|^s \Delta_j f \|_{L^p(\mathbb T^d)} \lesssim 2^{js} \| f \|_{L^p(\mathbb T^d)}.
\end{align*}
By using a fattened projection $\tilde \Delta_j= \sum_{l=-2}^2 \Delta_{j-l}$ (and noting
that $\Delta_j f = \tilde \Delta_j \Delta_j f$), one can then derive \eqref{b_per_e1}.

Next we derive \eqref{b_per_e2}. By Young's inequality, we have
\begin{align*}
\| \Delta_j f \|_{L^q(\mathbb T^d)} \lesssim \| \tilde \psi_j \|_{L^r(\mathbb T^d)} \| f\|_{L^p(\mathbb T^d)},
\end{align*}
where $\frac 1r =1+\frac 1q - \frac 1 p$. By \eqref{per_delta} and $\hat f(0)=0$, easy to check that $\Delta_j f =0$ if $j< -2$.
Therefore we may assume without loss of generality that $j\ge -2$. Then by using the fact that $\psi$
is Schwartz, we get
\begin{align*}
 & \| \sum_{k \in \mathbb Z^d } 2^{jd} \psi(2^j (z+2\pi k)) \|_{L_z^r(\mathbb T^d)} \notag \\
 \lesssim & \sum_{|k|\le 100} 2^{jd} \| \psi(2^j (z+2\pi k) ) \|_{L_z^r(\mathbb T^d)}
  + \sum_{|k|>100} 2^{jd} \langle 2^j k \rangle^{-100d} \notag \\
\lesssim &\; 2^{jd} \| \psi(2^j z) \|_{L_z^r(\mathbb R^d)} +1
\lesssim \; 2^{jd} 2^{-j \frac d r}.
 \end{align*}
 Thus \eqref{b_per_e2} is proved.
\qquad\end{proof}

\section{Proof of Proposition \ref{prop_max_0}}
For $0\le t \le T$, consider
$\displaystyle
f(t,x) = |\nabla h(t,x)|^2.
$
Note that
\begin{align*}
\partial_t h = (f-1) \Delta h + \nabla f \cdot \nabla h.
\end{align*}
Clearly
$\displaystyle
\partial_t \nabla h = (\Delta h ) \nabla f + (f-1) \Delta \nabla h
+ \sum_{j=1}^d \partial_j \nabla h \partial_j f + \sum_{j=1}^d \partial_j h \partial_j \nabla f.
$

Therefore
\begin{align}
\frac 12 \partial_t f & = \nabla h \cdot \partial_t \nabla h \notag
\\
& = \Delta h ( \nabla h \cdot \nabla f) + (f-1) (\nabla h )
\cdot(\Delta \nabla h ) + \sum_{j=1}^d (\nabla h \cdot \partial_j \nabla h)
\partial_j f \notag \\
& \qquad+ \sum_{j=1}^d \partial_j h ( \nabla h \cdot
\partial_j \nabla f) \notag \\
& = \Delta h (\nabla h \cdot \nabla f) + (f-1) (\nabla h)
\cdot(\Delta \nabla h) + \frac 12 |\nabla f |^2 + \sum_{j,k=1}^d \partial_j h \partial_k h \partial_{jk} f.
\label{e_max_2}
\end{align}

By definition, it is easy to check that
\begin{align*}
\Delta f = 2 \nabla \Delta h \cdot \nabla h + 2 \sum_{k,j=1}^d
(\partial_k \partial_j h)^2.
\end{align*}

Therefore
$\displaystyle
\nabla h \cdot \Delta \nabla h = \frac 1 2 \Delta f -\sum_{k,j=1}^d (\partial_k \partial_j h)^2.
$

Plugging this expression into \eqref{e_max_2}, we then obtain
\begin{align}
\frac 12 \partial_t f & = \frac 12 (f-1)\Delta f - (f-1)\sum_{k,j=1}^d (\partial_k \partial_j h)^2
+\Delta h (\nabla h \cdot \nabla f) \notag \\
& \qquad + \frac 12 |\nabla f|^2 + \sum_{k, j=1}^d \partial_j h \partial_k h \partial_j \partial_k f.
\notag
\end{align}

Now let $\epsilon>0$ be a small parameter which will tend to zero later. Consider the auxiliary
function
\begin{align*}
f^{\epsilon}(t,x) = f(t,x) - \epsilon t, \quad \forall\, 0\le t\le T, \, x \in \mathbb T^d.
\end{align*}

Note the equation for $f^{\epsilon}$ reads as
\begin{align}
\frac 12 \partial_t f^{\epsilon} & = -\frac 12 \epsilon +
\frac 12 (f^{\epsilon}+\epsilon t-1)\Delta f^{\epsilon} -
(f^{\epsilon}+\epsilon t-1)\sum_{k,j=1}^d (\partial_k \partial_j h)^2
\notag \\
&\quad\;+\Delta h (\nabla h \cdot \nabla f^{\epsilon})
+ \frac 12 |\nabla f^{\epsilon} |^2
+ \sum_{k, j=1}^d \partial_j h \partial_k h \partial_j \partial_k f^{\epsilon}.
\label{e_max_4}
\end{align}

Since $f^{\epsilon}$ is a continuous function on the compact domain $[0,T]\times \mathbb T^d$,
it must achieve its maximum at some point $(t_*,x_*)$, i.e.
\begin{align*}
\max_{0\le t \le T,\, x \in \mathbb T^d} f^{\epsilon}(t,x) = f^{\epsilon}(t_*,x_*)=: M_{\epsilon}.
\end{align*}

We discuss several cases.

Case 1. $0<t_*\le T$ and $M_{\epsilon}>1$. In this case observe that
\begin{align*}
&\nabla f^{\epsilon}(t_*,x_*)=0, \quad
 \Delta f^{\epsilon}(t_*,x_*)\le 0, \notag \\
& \sum_{k,j=1}^d c_j c_k (\partial_{j} \partial_k f^{\epsilon})(t_*,x_*)\le 0, \quad
\text{for any $(c_1,\cdots,c_d)\in \mathbb R^d$.}
\end{align*}

Therefore by \eqref{e_max_4} and the fact that $M_{\epsilon} >1$, we have
\begin{align*}
\frac 12 (\partial_t f^{\epsilon})(t,x) \Bigr|_{(t_*,x_*)} &
\le -\frac 12 \epsilon +\frac 12 (M_{\epsilon} +\epsilon t_*-1) (\Delta f^{\epsilon})(t_*,x_*) \notag \\
& \qquad -(M_{\epsilon}+\epsilon t-1) \sum_{k,j=1}^d (\partial_k \partial_j h)^2
\notag \\
& \le -\frac 12 \epsilon <0.
\end{align*}

This obviously contradicts to the fact that $0<t_*\le T$ and $(t_*,x_*)$ is a maximum.
Hence Case 1 is impossible.

Case 2. $0<t_*\le T$ and $M_{\epsilon} \le 1$. In this case we obtain the bound
\begin{align*}
\max_{0\le t \le T,\, x \in \mathbb T^d} f(t,x) \le \epsilon T +1.
\end{align*}

Case 3. $t_*=0$. Clearly then
\begin{align*}
\max_{0\le t \le T, \, x\in \mathbb T^d} f(t,x) \le  \max_{x\in \mathbb T^d} f(0,x)+ \epsilon T.
\end{align*}

Concluding from all cases and sending $\epsilon$ to zero, we obtain \eqref{prop00_e0}.

In the case dimension $d=1$, the proof of \eqref{prop00_e0a} is
similar. Set $g=h_x$. Note that
\begin{align*}
\partial_t g  = ( g^3-g)_{xx} = (3g^2-1) g_{xx} + 6 g  (g_x)^2.
 \end{align*}
Clearly  $(3g^2-1)g_{xx}$  is elliptic when  $3g^2> 1$, whence
\begin{align*}
\| g(t)\|_{\infty} \le \max\{\|g(0)\|_{\infty}, \frac 1 {\sqrt 3}
\}, \qquad \forall\, t\ge 0.
\end{align*}

\section{Proof of Theorem \ref{thm_lwp}}
\begin{lem} \label{lem_L_1}
Let $\nu>0$ and $L=-\nu \Delta^2 $.  Then for any integer $m\ge 1$
and any $t>0$, we have
\begin{align}
 &\|  D^m e^{tL} f \|_{L_x^{\infty}(\mathbb T^d)}
 \lesssim_{\nu,d,m}  (1+t^{-\frac m 4})
 \|f\|_{H^{d/2}_x(\mathbb T^d)}; \label{esec3_01a}
 \end{align}
Similarly for any integer $m\ge 0$ and any $t>0$,
 \begin{align}
  & \|D^m e^{tL} f \|_{L_x^\infty(\mathbb T^d)} \lesssim_{\nu,d,m} t^{-\frac m 4}
  \|f\|_{L_x^{\infty}(\mathbb T^d)}, \label{esec3_01b} \\
  & \|D^m e^{tL} f \|_{L_x^2(\mathbb T^d) } \lesssim_{\nu,d,m} (1+t^{-\frac m 4})\|f \|_{L_x^2(\mathbb T^d)}.
  \label{esec3_01c}
 \end{align}

 In the above $D^m$ denotes any differential operator of order $m$. For
 example $D^2$ can be any  one of the operators $\partial_{x_i x_j}$, $1\le i,j\le d$.

 If $f$ has mean zero, then \eqref{esec3_01a} and \eqref{esec3_01c}
 can be improved as:
 \begin{align}
 & \| D^m e^{tL} f \|_{\infty} \lesssim_{\nu,d,m} t^{-\frac m4} \| f
 \|_{H^{\frac d 2}}, \quad \forall\, m\ge 1,\, t>0, \label{esec3_01d} \\
 & \|D^m e^{tL} f ||_{2} \lesssim_{\nu,d,m} t^{-\frac m4} \|f\|_2, \qquad \forall\, m\ge 0, \, t>0.
 \label{esec3_01e}
 \end{align}
\end{lem}

\begin{proof} 
We first show \eqref{esec3_01a}. Define $\langle \nabla \rangle = \sqrt{1-\Delta}$.
 Clearly
\begin{align*}
D^m e^{tL} f = D^m e^{tL} \langle\nabla\rangle^{-\frac d2} \langle
\nabla \rangle^{\frac d 2} f= K_1 \ast (\langle \nabla\rangle^{\frac
d 2} f)
\end{align*}
where $\ast$ denotes the usual convolution and $K_1$ is the kernel
corresponding to $D^m e^{tL} \langle \nabla \rangle^{-\frac d2}$.
Then
\begin{align*}
\|D^m e^{tL} f \|_{L_x^{\infty}(\mathbb T^d)} \lesssim \|
K_1\|_{L_x^2(\mathbb T^d)} \|f\|_{H^{\frac d 2}_x(\mathbb T^d)}.
\end{align*}
Now since $m\ge 1$,
\begin{align*}
\| K_1\|_{L_x^2}^2 \lesssim \sum_{k \in \mathbb Z^d} e^{-2 \nu
t|k|^4 }
|k|^{2m} \cdot \langle k \rangle^{-d}
 \lesssim 1+ \sum_{ k \ne 0} e^{-2\nu t  |k|^4}
|k|^{2m-d}
 \lesssim 1+ t^{-\frac m2}.
\end{align*}
Thus \eqref{esec3_01a} follows easily.

For \eqref{esec3_01b}, we can regard $f$ as a periodic function on
$\mathbb R^d$. Then using the fact that for any multi-index $\alpha$
with $|\alpha|=m$,  $\|\mathcal F^{-1} ( \xi^\alpha e^{-t |\xi|^4} )
\|_{L_x^1(\mathbb R^d)} \lesssim t^{-\frac m4}$, we get
\begin{align*}
\| D^m e^{tL} f \|_{L_x^{\infty}(\mathbb T^d)}  = \|D^m e^{tL} f
\|_{L_x^{\infty}(\mathbb R^d)}
\lesssim t^{-\frac m 4} \| f \|_{L_x^{\infty}(\mathbb R^d)}
 \lesssim t^{-\frac m 4} \|f \|_{L_x^{\infty}(\mathbb T^d)}.
\end{align*}

Similarly one can prove \eqref{esec3_01c} by computing everything on
the Fourier side.

In the case $f$ has mean zero, we note that $\hat f(0)=0$, and
\eqref{esec3_01d}--\eqref{esec3_01e} follows easily.
\qquad\end{proof}

\subsection*{Proof of Theorem \ref{thm_lwp}}
This is more or less a standard application of the theory of mild
solutions. Therefore we shall only sketch the details.

We recast \eqref{0} into the mild form (alternatively one
can also construct the mild solution by considering
$L=-\nu\Delta^2-\Delta $ as the linear part and  taking $e^{tL}$ as
the linear propagator):
\begin{align*}
h(t)&= e^{-t\nu \Delta^2}  h_0 + \sum_{j=1}^d \int_0^t \partial_j
e^{-(t-s)\nu \Delta^2} ( (|\nabla h|^2-1) \partial_j h )(s) ds
\notag \\
&=: e^{-t \nu \Delta^2}h_0 + \Phi(h)(t).
\end{align*}

Fix $h_0 \in H^{d/2}(\mathbb T^d)$.
Define $h^{(0)} = e^{-t \nu \Delta^2} h_0$, and for $j\ge 1$,
\begin{align*}
h^{(j)}(t)= e^{-t\nu \Delta^2}  h_0 + \Phi(h^{(j-1)})(t).
\end{align*}

For $T>0$, introduce the Banach space
\begin{align*}
X_T= \Bigl\{ h \in C_t^0 H_x^{\frac d 2} ([0,T]\times \mathbb T^d):\;
t^{\frac 14} \nabla h \in C_t^0 C_x^0, \, t^{\frac 14} h \in C_t^0 H_x^{\frac d2+1}\Bigr\}
\end{align*}
with the norm
\begin{align*}
\| h \|_{X_T} = \|h\|_{C_t^0 H_x^{\frac d2}} + \|t^{\frac 14} \nabla h \|_{L_{t,x}^{\infty}}
+ \| t^{\frac 14} h \|_{C_t^0 H_x^{\frac d2+1}}.
\end{align*}
For convenience denote the seminorm
\begin{align*}
\|h\|_{Y_T}= \|t^{\frac 14} \nabla h \|_{L_{t,x}^{\infty}}
+ \| t^{\frac 14} h \|_{C_t^0 H_x^{\frac d2+1}}.
\end{align*}

We shall show that for sufficiently small $T>0$ (depending on the profile of $h_0$),
the iterates $h^{(j)}$, $j\ge 0$ form a Cauchy sequence in the set
\begin{align*}
B_T = \{ h \in X_T:\, \|h\|_{X_T} \le 2  \|h_0\|_{H^{\frac d2}(\mathbb T^d)}, \,
\|h\|_{Y_T} \le 2\epsilon_1 \| h_0\|_{H^{\frac d2}(\mathbb T^d)}\},
\end{align*}
where $\epsilon_1>0$ is a sufficiently small constant depending only on $(\nu,d)$ and $\|h_0\|_{H^{\frac d2}}$.

We shall only verify that $h^{(j)} \in B_T$ and omit the contraction argument since it is
quite similar.

Consider first $j=0$. For $h_0 \in H^{\frac d2}(\mathbb T^d)$, obviously
\begin{align*}
\| e^{-\nu \Delta^2 t } h_0 \|_{C_t^0 H_x^{\frac d2}} \le \| h_0\|_{H^{\frac d2}}.
\end{align*}
By Lemma \ref{lem_L_1} and a density argument, we have for $h_0 \in H^{\frac d2}$,
\begin{align*}
\lim_{t \to 0+} \| t^{\frac 14} \nabla e^{-\nu t \Delta^2 }  h_0\|_{L_{x}^{\infty}} =0,
\quad \lim_{t\to 0+} \| t^{\frac 14} e^{- \nu t \Delta^2} h_0\|_{H_x^{\frac d2+1}}=0.
\end{align*}
Thus for $T>0$ sufficiently small,
\begin{align*}
&\| h^{(0)} \|_{X_T} \le \frac 32 \|h_0\|_{H^{\frac d2}},\quad
\|h^{(0)}\|_{Y_T} \le \epsilon_1 \|h_0\|_{H^{\frac d2}},
\end{align*}
where $\epsilon_1$ will be taken sufficiently small (depending on $(\nu,d)$ and $\|h_0\|_{H^{\frac d2}}$) later
when we verify the estimates for $h^{(j)}$, $j\ge 1$.

Now inductively assume $h^{(j-1)}\in B_T$.
To show $h^{(j)} \in B_T$, it suffices for us to check
\begin{align*}
\| \Phi( h^{(j-1)} ) \|_{X_T} \le \epsilon_1 \|h_0\|_{H^{\frac d2}}.
\end{align*}
To simplify notation, in the computation below
we shall drop the superscript $(j-1)$ and write $\Phi(h^{(j-1)})$ simply as $\Phi(h)$. We also
write $\lesssim_{\nu,d}$ simply as $\lesssim$.

{Note that without loss of generality we can assume $t\lesssim 1$, so that when
applying Lemma \ref{lem_L_1}, we have $1+t^{-\frac m4} \lesssim t^{-\frac m4}$ (i.e.
the constant $1$ is not needed).}
Now by Lemma \ref{lem_L_1}, we have
\begin{align*}
\| \Phi(h)(t) \|_{H_x^{\frac d 2}} & \lesssim \Bigl\|\int_0^t \langle \nabla \rangle^{\frac d2}
\nabla \cdot e^{-(t-s)\nu \Delta^2}\bigl( |\nabla h|^2-1) \nabla h \bigr) (s) ds \Bigr\|_2 \notag \\
& \lesssim \int_0^t (t-s)^{-\frac 14} \| \langle \nabla \rangle^{\frac d2} \nabla h (s) \|_2 ds \notag\\
& \qquad + \int_0^t (t-s)^{-\frac 14} \| \langle \nabla \rangle^{\frac d2} \bigl( |\nabla h(s)|^2 \nabla h(s) \bigr)
\|_2 ds \notag \\
& \lesssim \int_0^t (t-s)^{-\frac 14} s^{-\frac 14} ds \cdot \|s^{\frac 14} h(s) \|_{C_s^0H_x^{\frac d2+1}} \notag \\
& \qquad + \int_0^t (t-s)^{-\frac 14} s^{-\frac 34} ds \cdot \| s^{\frac 14} h (s)\|_{C_s^0 H_x^{\frac d2+1}}
\cdot \|s^{\frac 14} \nabla h (s) \|_{L_s^{\infty}L_x^{\infty}}^2 \notag \\
& \lesssim  t^{\frac 12}  \|s^{\frac 14} h(s) \|_{C_s^0H_x^{\frac d2+1}}
 + \| s^{\frac 14} h (s)\|_{C_s^0 H_x^{\frac d2+1}}
\cdot \|s^{\frac 14} \nabla h (s) \|_{L_s^{\infty}L_x^{\infty}}^2 \notag \\
& \lesssim t^{\frac 12} \|h_0\|_{H^{\frac d2}} + \|h\|_{Y_t}^3.
\end{align*}
Thus for $T>0$ sufficiently small and $\epsilon_1$ sufficiently small,
\begin{align*}
\|\Phi(h) \|_{C_t^0 H_x^{\frac d2}([0,T]\times \mathbb T^d)} \le \frac{\epsilon_1}{10} \|h_0\|_{H^{\frac d2}}.
\end{align*}
Similarly easy to check that
\begin{align*}
\|t^{\frac 14} \Phi(h)(t)\|_{C_t^0 H_x^{\frac d2+1}([0,T]\times \mathbb T^d)}
+ \|t^{\frac 14} \nabla \Phi(h)(t)\|_{L_{t,x}^{\infty}([0,T]\times \mathbb T^d)}
\le  \frac {\epsilon_1} 5 \|h_0\|_{H^{\frac d2}}.
\end{align*}
Thus
\begin{align*}
\|\Phi(h) \|_{X_T} \le \epsilon_1 \| h_0\|_{H^{\frac d2}}.
\end{align*}
We have finished the proof of existence and uniqueness of a solution in the Banach space $X_T$.

The smoothing estimate of $h(t)$ for $t>0$ is utterly standard. For example if we know
$h \in L_t^{\infty} H_x^m ( [t_0,t_1]\times \mathbb T^d)$ on some time interval $[t_0,t_1]$, then
for $t\in (t_0,t_1]$,
\begin{align*}
 & \Bigl\|D^{m+1}\int_{t_0}^t \nabla \cdot e^{-(t-s) \nu \Delta^2} ( (|\nabla h |^2-1) \nabla h)(s) ds \Bigr\|_2
 \notag \\
 \lesssim & \int_{t_0}^t (t-s)^{-\frac 34} \|( |\nabla h(s) |^2-1)\nabla h(s)\|_{H^{m-1} } ds \notag \\
 \lesssim & \int_{t_0}^{t_1} (t-s)^{-\frac 34} ds\cdot \| h\|_{L_s^{\infty} H_x^m} \notag \\
 & \quad + \int_{t_0}^{t_1} (t-s)^{-\frac 34} s^{-\frac 12} ds
 \cdot \| h\|_{L_s^{\infty} H_x^m} \cdot \| s^{\frac 14} \nabla h \|_{L_s^{\infty} L_x^{\infty}}^2. \notag
 \end{align*}
 This shows that $h$ has higher regularity $H_x^{m+1}$ on $(t_0,t_1]$ (  the linear part
 $$e^{-(t-t_0) \nu \Delta^2}h(t_0) \in H_x^{m+1}$$ only for $t\in (t_0,t_1]$ ). We omit further details.


\section{Proof of Corollary \ref{cor_gwp} and Corollary
\ref{cor_grad}} \label{sec_refined_1}

\subsection*{Proof of Corollary \ref{cor_gwp}}
Let the dimension $d\le 3$.

We first assume that the initial data $h_0 \in H^{4}(\mathbb T^d)$ with mean zero. Denote the corresponding
solution obtained by Theorem \ref{thm_lwp} as $h$.  {To bound $\| \partial_t h \|_2$,
we need to control $\| \partial^2 h \cdot \partial h \cdot \partial h \|_2 \lesssim \| \partial^2 h\|_2
\| \partial h\|_{\infty}^2 \lesssim \|h\|_{H^4}^2.$ The $H^4$ regularity is used to control $\|\nabla h \|_{\infty}$.
} It is then easy to check that
$h\in C_t^0 H_x^4\cap C_t^1 L_x^2$ and
\begin{align} \label{pf_cor_gwp_e1}
\frac d {dt} E = -\| \partial_t h\|_2^2,
\end{align}
where
\begin{align*}
E(t) = \frac 12 \nu \|\Delta h(t) \|_2^2 + \frac 14 \int_{\mathbb T^d} (|\nabla h(t)|^2-1)^2 dx.
\end{align*}
Alternatively to avoid the issue of differentiability, one
can interpret \eqref{pf_cor_gwp_e1} as the integral formulation:
$\displaystyle E(t_2)=E(t_1)-\int_{t_1}^{t_2} \| \partial_t h\|_2^2dt$
for any $0\le t_1<t_2$.

From energy conservation we get $\|h(t) \|_{H^2} \lesssim \|h_0 \|_{H^2}$ for any $t>0$. Now for $H^2$ initial data (recall
the critical space in Theorem \ref{thm_lwp} is $H^{d/2}$ and $d/2<2$ for $d\le 3$),
 the lifespan of the local solution depends on the $H^2$-norm of the initial data. Thanks to this fact and
 the estimate $\|h(t)\|_{H^2} \lesssim \|h_0\|_{H^2}$, the corresponding local solution can be continued
 for all time by a standard argument. This concludes the proof of global wellposedness under
  the assumption that $h_0 \in H^4$.

Now let $h_0 \in H^{\frac d2}(\mathbb T^d)$ with mean zero.
By Theorem \ref{thm_lwp}, there exists a local solution $h$ on $[0,T_0]$ for some $T_0>0$ depending on $h_0$.
Let $h_1=h(T_0/2)$. By Theorem \ref{thm_lwp}, $h_1 \in H^m$ for all $m\ge 1$. In particular $h_1 \in H^4$.
Now with $h_1$ as initial data, the corresponding solution can be denoted as $\tilde h(t) = h(t+T_0/2)$. One
can then repeat the argument described in the previous paragraph to obtain global wellposedness.


\subsection*{Proof of Corollary \ref{cor_grad}}

\underline{The 1D case}. Note that by energy law we have $E(t) \le E_0$. Thus
\begin{align*}
\| \partial_{xx} h(t) \|_2 \lesssim \frac 1 {\sqrt {\nu}} \sqrt {E_0}, \quad
\| \partial_x h (t) \|_4 \lesssim E_0^{\frac 14} +1.
\end{align*}
By using the Gagliardo-Nirenberg interpolation inequality, we have
\begin{align*}
\| \partial_x h \|_{\infty} \lesssim \| \partial_x h \|_4^{\frac 23} \| \partial_{xx} h \|_2^{\frac 13}.
\end{align*}
Therefore
\begin{align*}
\| \partial_x h (t) \|_{\infty} \lesssim \nu^{-\frac 16} E_0^{\frac 16} (E_0^{\frac 16}+1).
\end{align*}

\underline{The 2D case}.
We first perform a short time estimate. Let $0<\epsilon<1$ which will be taken sufficiently small.
Consider
\begin{align*}
h(t) =e^{-\nu t \Delta^2} h_0 + \int_0^t  \nabla \cdot e^{-\nu (t-s) \Delta^2}
(|\nabla h|^2-1) \nabla h(s) ds.
\end{align*}
Easy to check that in 2D,
 $\| |\nabla|^{1+\frac{\epsilon}{100}} h \|_{\infty} \lesssim \| |\nabla|^{2+2\epsilon} h \|_{2-\epsilon}$ (recall $h$ has mean zero).
 Then
 \begin{align*}
 \| |\nabla|^{2+2\epsilon} h(t) \|_{2-\epsilon}
 & \lesssim \| |\nabla|^{2\epsilon} e^{-\nu t\Delta^2} |\nabla|^2 h_0 \|_{2-\epsilon} \notag \\
&\quad  + \int_0^t \| |\nabla|^{2+2\epsilon} \nabla \cdot e^{-\nu (t-s) \Delta^2}
 ( (|\nabla h|^2-1) \nabla h)(s) ds \|_{2-\epsilon} ds \notag \\
 & \lesssim (\nu t)^{-2\epsilon} \| h_0\|_{H^2} + \int_0^t ( \nu (t-s))^{-\frac{3+2\epsilon} 4}
 (\| h(s) \|_{H^2}^3 + \| h(s)\|_{H^2} ) ds \notag \\
 & \lesssim (\nu t)^{-2\epsilon} (\frac{E_0}{\nu})^{\frac 12} + \nu^{-\frac{3+2\epsilon}4} t^{\frac{1-2\epsilon}4}
 ( (\frac{E_0}{\nu})^{\frac 12} + (\frac{E_0}{\nu})^{\frac 32} ).
 \end{align*}
 In the above when bounding the nonlinearity, we used the estimate
 \begin{align*}
 \| |\nabla h |^2 \nabla h \|_{2-\epsilon} \lesssim \| \nabla h \|_2 \| \nabla h \|_{\frac{2-\epsilon} {\epsilon} }^2
 \lesssim \| h \|_{H^2}^3.
 \end{align*}
Thus for $t\sim 1$ and $0<\nu \lesssim 1$, we get
\begin{align*}
\| |\nabla|^{1+\frac{\epsilon}{100}} h(t) \|_{\infty} \lesssim ( \frac{E_0+1} {\nu} )^{10}.
\end{align*}
By repeating the same analysis with $t\gg 1$ and $h_0$ replaced by $h(t-1)$ (note that only $\|h\|_{H^2}$ enters
the analysis), we get for all $t\gtrsim 1$
\begin{align*}
\| |\nabla|^{1+\frac{\epsilon}{100}} h (t) \|_{\infty} \lesssim ( \frac{E_0+1} {\nu} )^{10}.
\end{align*}
Now note that $\| h(t) \|_{H^2} \lesssim (\frac {E_0}{\nu} )^{\frac 12}$.
Using Littlewood-Paley decomposition (note that $S_{-2} \nabla h =0$),
we get
\begin{align*}
\| \nabla h(t) \|_{L^\infty(\mathbb T^2)} & \lesssim
\sum_{-2\le j\le j_0} \| \Delta_j \nabla  h \|_{L^{\infty}(\mathbb T^2)} +
\sum_{j>j_0} \|\Delta_j \nabla  h\|_{L^{\infty}(\mathbb T^2)}
\notag \\
& \lesssim (j_0+3) \| h\|_{H^2} +2^{-j_0\frac{\epsilon}{100}} \| |\nabla|^{1+\frac{\epsilon}{100}} h \|_{\infty}
\notag \\
& \lesssim (j_0+3) (\frac{E_0}{\nu})^{\frac12}+2^{-j_0 \frac{\epsilon}{100}} (\frac{E_0+1}{\nu} )^{10}.
\end{align*}
Optimizing in $j_0$, we get
\begin{align*}
\sup_{1\lesssim t<\infty} \| \nabla h(t)\|_{\infty} \lesssim ( \frac{E_0}{\nu})^{\frac 12} |\log( \frac{E_0+1}{\nu})|.
\end{align*}

Now to obtain the estimate for $t\lesssim 1$, we simply note that for $t\lesssim 1$,
by repeating the analysis before,
\begin{align*}
\| |\nabla|^{1+\frac{\epsilon}{100}} ( h(t) -e^{-\nu t \Delta^2} h_0 ) \|_{\infty}
\lesssim \left( \frac{E_0+1} {\nu} \right)^{10}.
\end{align*}
On the other hand,
\begin{align*}
\| h(t) -e^{-\nu t \Delta^2} h_0 \|_{H^2} \lesssim \|h\|_{H^2}+\|h_0\|_{H^2}
\lesssim ( \frac {E_0}{\nu} )^{\frac 12}.
\end{align*}
Thus we obtain the same bound for $h(t) -e^{-\nu t \Delta^2} h_0$.

This finishes the estimate for the 2D case.

\underline{The 3D case}.
We shall again perform a short time estimate. Write
\begin{align*}
\nabla h (t) = e^{-\nu t \Delta^2 } \nabla h_0
+ \int_0^t \nabla \nabla \cdot e^{-\nu (t-s) \Delta^2} ( (|\nabla h|^2-1) \nabla h)(s) ds.
\end{align*}
It is easy to check that
\begin{align*}
\| e^{- \nu \Delta^2 t } \nabla h_0 \|_{L_x^{\infty}(\mathbb T^3)} \lesssim (\nu t)^{-\frac 18}
\| h_0 \|_{H_x^2(\mathbb T^3)}.
\end{align*}
We then get for $t\lesssim 1$,
\begin{align*}
\| \nabla h(t) \|_{\infty}
& \lesssim (\nu t)^{-\frac 18} \|h_0\|_{H^2}
 + \int_0^t  (\nu (t-s) )^{-\frac 7 8}  ( \| \nabla h(s) \|_6^3 + \|\nabla h(s) \|_2) ds \notag \\
& \lesssim t^{-\frac 18} \nu^{-\frac 58} E_0^{\frac 12}  + \nu^{-\frac 78} t^{\frac 18} (\nu^{-\frac 32} E_0^{\frac 32} +1).
\end{align*}
Choosing $t\sim \nu^7$ then yields
$\displaystyle
\| \nabla h(t) \|_{\infty} \lesssim \nu^{-\frac 32} (E_0^{\frac 32}+1).
$
For general $t\gg \nu^7$, we can replace $h_0$ by $h(t-\nu^7)$ and repeat the above analysis.
This ends the estimate for the 3D case. \qquad\endproof
\vspace{.1in}


The following proposition shows that in 1D, there exists initial data such that the corresponding solution
obeys uniform in time gradient bounds which are independent of $\nu$.

\begin{prop} \label{prop_refine_1D}
Let the dimension $d=1$. Consider \eqref{2} on the
$2\pi$-periodic torus $\mathbb T$ with $0<\nu\lesssim 1$.
Assume $h_0 \in H^2(\mathbb T)$ with mean zero and let $h=h(t,x)$ be
the corresponding global solution to \eqref{2}.
Denote
\begin{align*}
E_0= \int_{\mathbb T^d} \Bigl( \frac 12 \nu |\partial_{xx} h_0|^2 + \frac 14 (|\partial_x h_0|^2-1)^2 \Bigr) dx.
\end{align*}
Then  for all $t>0$ and some absolute constant
$C_1>0$,
\begin{align} \label{prop_ref1_e1}
&\| \partial_x h(t) \|_{\infty} \le C_1 \max\{1, \nu^{-\frac 16} E_0^{\frac 13}\}.
\end{align}

For each $0<\nu \lesssim 1$, there exists a family $\mathcal A_{\nu}$ of initial data, such that
if $h_0 \in \mathcal A_{\nu}$, then $E_0 \lesssim \sqrt{\nu}$, and the corresponding solution
satisfies
\begin{align*}
\| \partial_x h(t) \|_{\infty} \le B_1, \qquad\forall\, t\ge 0,
\end{align*}
where $B_1>0$ is an absolute constant. (In particular, it is independent of $\nu$).

\end{prop}

\subsection*{Proof of Proposition \ref{prop_refine_1D}}
We first show \eqref{prop_ref1_e1}. Denote $\|h_x\|_{\infty}=A$ and $g=h_x^2-1$.  If $A\le 2$ we are done.
Now assume $A>2$, then obviously $A^2 \lesssim \|g\|_{\infty}$. Now by Gagliardo-Nirenberg interpolation, we get
\begin{align*}
A^2\lesssim \| g \|_{\infty}  \lesssim \| g \|_2^{\frac 12} \| \partial_x g \|_2^{\frac 12}
 \lesssim \| g\|_2^{\frac 12} \| \partial_{xx} h\|_2^{\frac 12} \| \partial_x h\|_{\infty}^{\frac 12}
 \lesssim \| g\|_2^{\frac 12} \| \partial_{xx} h \|_2^{\frac 12} A^{\frac 12}.
\end{align*}
Thus
\begin{align*}
A \lesssim \| g \|_2^{\frac 13} \| \partial_{xx} h \|_2^{\frac 13} \lesssim E_0^{\frac 16} (\frac{E_0}{\nu})^{\frac 16}  \lesssim \nu^{-\frac 16} E_0^{\frac 13}.
\end{align*}

We now show that there exists initial data $h_0$ such that
$\displaystyle
E_0 \lesssim \sqrt {\nu}.
$
The idea is to mollify the \enquote{sawtooth}-type profile and add a $\delta$-cap ($\delta\approx \sqrt{\nu}$)
around each tips of the sawtooth. To this end, let $L_0\ge 3$ be an integer and  define
\begin{align*}
g_0(x) = \int_0^x \operatorname{sgn}( \sin (L_0 \tau) ) d\tau, \qquad x \in [-\pi,\pi],
\end{align*}
where $\operatorname{sgn}$ is the usual sign function:
\begin{align*}
\operatorname{sgn}(z)=\begin{cases}
1,\quad z>0,\\
0, \quad z=0,\\
-1, \quad z<0.
\end{cases}
\end{align*}
The value of $L_0$ is not important as long as it is independent of $\nu$.

Now around each local maxima or minima of $g_0$, easy to check that $g_0^{\prime}$ change its sign from
$-1$ to $1$, or $1$ to $-1$. At the maxima (minima), $g_0^{\prime}$ is undefined. One can then mollify
$g_0$ therein within a $\delta$-neighborhood. Denote the mollified function as $g_{\delta}$. Then
\begin{align*}
E(g_{\delta})  = \int_{\mathbb T} \Bigl( \frac 12 {\nu} |\partial_{xx} g_{\delta}|^2+\frac 14
(|\partial_x g_{\delta} |^2-1)^2 \Bigr) dx  \lesssim_{L_0} \nu \cdot \frac 1 {\delta^2} \cdot \delta+ \delta.
\end{align*}
Choosing $\delta\sim \sqrt{\nu}$ then yields $E(g_{\delta}) \lesssim_{L_0} \sqrt{\nu}$. \qquad \endproof


\begin{prop} \label{prop_refine_1D_long}
Let the dimension $d=1$. Consider \eqref{2} on the
$2\pi$-periodic torus $\mathbb T$ with $0<\nu\lesssim 1$.
Assume $h_0 \in H^{\frac 12}(\mathbb T)$ with mean zero and let $h=h(t,x)$ be
the corresponding global solution to \eqref{2}. Then
\begin{align} \label{pr1D_e00}
\limsup_{t\to \infty} \| \partial_x h(t) \|_{\infty} \le K_0,
\end{align}
where $K_0$ is a constant depending only on the initial data $h_0$.
If in additional $h_0$ is even in $x$, then \eqref{pr1D_e00} can be improved to
\begin{align} \label{pr1D_e01}
\limsup_{t\to \infty} \| \partial_x h(t) \|_{\infty} \le 1.
\end{align}
\end{prop}
\begin{rem}
Recall that in the 1D case, the equation \eqref{2} can be transformed into the usual
Cahn-Hilliard equation via the change of variable $u=\partial_x h$. The convergence
to steady states (and consequently gradient bounds) can be obtained using
the {\L}ojasiewicz-Simon inequality (cf. \cite{RH99}). Our proof below however does not appeal
to this theory and gives an alternative approach.
\end{rem}

\subsection*{Proof of Proposition \ref{prop_refine_1D_long}}
First observe that by using Theorem \ref{thm_lwp} and a shift in time
we may assume $h_0 \in H^{10}(\mathbb T)$. By using the Duhamel formula
\begin{align*}
h(t) = e^{-\nu t\partial_{x}^4} h_0 + \int_0^t e^{-\nu (t-s) \partial_x^4} \partial_x( (h_x^2-1)h_x)(s) ds,
\end{align*}
the energy law, and the exponential (in time) decay of the propagator $e^{-\nu (t-s) \partial_x^4}$
(acting on mean-zero functions), it is not difficult to derive that
\begin{align} \label{pr1dL_e3}
\sup_{t\ge 0} \| h(t) \|_{H^{10}(\mathbb T)} \lesssim_{\nu, E_0} 1.
\end{align}
This estimate will be used below.

\underline{Step 1}: we show that $\lim_{t\to \infty} \| \partial_t h \|_{\infty} =0$.
Denote $g=\partial_t h$, then $g$ satisfies
the equation
$\displaystyle
\partial_t g = \partial_x ( (3h_x^2-1) g_x) - \nu \partial_x^4 g.
$
Consider $t>t_0$, where $t_0$ will be picked later. We have
\begin{align}
g(t) &= e^{-\nu (t-t_0) \partial_x^4} g(t_0) + \int_{t_0}^t \partial_x e^{-\nu (t-s) \partial_x^4}
( (3 h_x^2-1) g_x) (s) ds \notag \\
& = e^{-\nu (t-t_0) \partial_x^4} g(t_0)
+ \int_{t_0}^t \partial_{xx} e^{-\nu(t-s) \partial_x^4} ( (3h_x^2-1) g)(s) ds \notag \\
& \qquad - \int_{t_0}^t \partial_x e^{-\nu(t-s) \partial_x^4} ( 6h_{xx}h_x g)(s) ds. \label{pr1dL_e5}
\end{align}

Now note that for any function $\tilde g:\, \mathbb T\to \mathbb R$ (not necessarily having mean zero), one
has for $m\ge 1$,
\begin{align*}
\| \partial_x^m e^{-\nu t \partial_x^4} \tilde g\|_2 \lesssim_{m,\nu} e^{-\nu t/100} t^{-\frac m4} \| \tilde
g \|_2.
\end{align*}
Here the point is that since $m\ge 1$, $\tilde g$ can be replaced by $\tilde g - \bar{\tilde g}$ ($\bar{\tilde g}$ denotes
the mean of $\tilde g$) and $|\bar{\tilde g}|\lesssim \|\tilde g\|_2$.

Now continuing from \eqref{pr1dL_e5}, we get (by using \eqref{pr1dL_e3})
\begin{align}
\| g(t) \|_2 &\lesssim_{\nu, E_0} \| g(t_0)\|_2 + \int_{t_0}^t (t-s)^{-\frac 12}
 e^{-\nu (t-s)/100}\|g(s)\|_2 ds \notag \\
 & \qquad\quad+\int_{t_0}^t (t-s)^{-\frac 14}
 e^{-\nu (t-s)/100}\|g(s)\|_2 ds. \label{pr1dL_e7}
\end{align}

By using the energy law, we have $\int_0^{\infty} \| g(s) \|_2^2 ds <\infty$. Thus one can find $t_0$ sufficiently
large such that $\|g(t_0)\|_2 \ll 1$ and also $\int_{t_0}^{\infty} \|g(s)\|_2^2 ds \ll 1$. By \eqref{pr1dL_e3}, we also
have $\sup_{s\ge 0} \|g(s)\|_2 \lesssim 1$. These estimates with \eqref{pr1dL_e7} and an
$\epsilon$-$\delta$ argument (One needs to split the time interval in \eqref{pr1dL_e7}. For $s$ close to $t$,
we use the smallness of the time interval and the estimate $\|g(s)\|_2 \lesssim 1$. For $s$ away from $t$, use
$\int_{t_0}^{\infty} \|g(s)\|_2^2ds \ll 1$.)
then easily yield
\begin{align*}
\lim_{t\to \infty} \|g(t)\|_2 =0.
\end{align*}
Interpolating the above estimate with \eqref{pr1dL_e3}
(recall $g(t)=\partial_t h = (h_x^3-h_x)_x-\nu \partial_x^4 h$), we get
\begin{align} \label{pr1d_e10}
\lim_{t\to \infty} \| \partial_t h \|_{\infty} =0.
\end{align}

\underline{Step 2}: we show \eqref{pr1D_e01}. Easy to check that  the even symmetry is propagated in
time. Denote $f=\partial_x h$. Then
\begin{align*}
\partial_x \Bigl(  f^3 -f -\nu f_{xx} \Bigr) = \partial_t h.
\end{align*}
In view of the even symmetry of $h$, we have $f(t,x=0)\equiv 0$, $\partial_{xx} f(t,x=0)\equiv0$. Thus
\begin{align*}
(f^2-1)f - \nu \partial_{xx}f = \int_0^x (\partial_t h)(t, y)dy.
\end{align*}
A simple maximum principle argument together with \eqref{pr1d_e10} then yield \eqref{pr1D_e01}.

Finally the proof of \eqref{pr1D_e00} is similar. In the general case, observe that (since $f=\partial_x h$)
\begin{align*}
\frac 1 {2\pi} \int_{\mathbb T} (f^3-f -\nu f_{xx}(t,x) ) dx =
\underbrace{\frac 1{2\pi} \int_{\mathbb T} f^3(t,x) dx}_{:=m(t)}.
\end{align*}
By the Mean Value Theorem, there exists $x_0\in[-\pi,\pi]$ such that
\begin{align*}
f^3(t,x_0) - f(t,x_0) -\nu f_{xx}(t,x_0) =m(t).
\end{align*}
We then have
\begin{align*}
f^3 - f - \nu f_{xx} = \int_{x_0}^x (\partial_t h)(t,y) dy + m(t).
\end{align*}
Now observe that
\begin{align*}
|m(t)| \lesssim \|\partial_x h(t)\|_3^3 \lesssim  1+\int_{\mathbb T} (h_x^2-1)^2 dx  \lesssim 1+E_0,
\end{align*}
where $E_0$ is the initial energy. The bound \eqref{pr1D_e00} then again follows from a maximum
principle argument using this estimate.


\section{Proof of Theorem \ref{thm0} and Corollary \ref{cor0}}

The following perturbation lemma is more or less standard. It follows from the local theory
and we omit the proof.

\begin{prop}[Finite time stability of solutions] \label{prop_stable}
Let $\nu>0$ in \eqref{0}. Let $u_0 \in H^k$, $k> d/2$ and $u$ be the corresponding
solution. Let $T>0$ be given and assume $u$ has lifespan bigger than $[0,T]$. Then for any $\epsilon>0$, there exists $\delta>0$
such that the following holds:

For any $v_0\in H^k$, $k>d/2$ with
$\displaystyle
\| v_0 -u_0\|_{H^k} <\delta,
$
there exists a solution $v$ to \eqref{0} corresponding to the initial data $v_0$ and has
lifespan containing $[0,T]$. Furthermore
we have
\begin{align*}
\max_{0\le t \le T} \| v(t) -u(t) \|_{H^k} <\epsilon.
\end{align*}
In particular by shrinking $\delta$ further if necessary, we have
\begin{align*}
\max_{0\le t \le T} \| \nabla v(t) - \nabla u(t) \|_{\infty} <\epsilon.
\end{align*}

\end{prop}

We now complete the proof of Theorem \ref{thm0}.

\begin{proof}[Proof of Theorem \ref{thm0}]
Step 1. We first show that there exists a smooth solution $w$ to \eqref{2} with initial data $w_0$ such that
$\displaystyle
\| w_0^{\prime}\|_{\infty} =1
$
and for some $t_*>0$, $C_1>1$
\begin{align}
\| \partial_x w(t_*) \|_{\infty}>C_1>1. \label{pf_thm0_1}
\end{align}

Let $\eta>0$ be sufficiently small and $w_0$ be a smooth $2\pi$-periodic function
with mean zero (Here one can choose $w_0$ such that it is odd in $x$ when regarded as
a function on $\mathbb R$. This in turn easily implies that $w_0$ has mean zero on $[-\pi,\pi]$.)
 such
that
\begin{align}
&w_0(x) = x- \eta x^5, \quad |x| <\eta, \notag \\
& | w_0^{\prime}(x)|<1, \quad \eta \le |\xi|\le \pi. \label{step1_1}
\end{align}

Denote by $w=w(t,x)$ the corresponding solution to \eqref{2}. Observe that
\begin{align*}
w_0^{\prime}(x)=1-5 \eta x^4, \quad \text{for $|x| <\eta$}.
\end{align*}
Obviously it follows that $|w_0^{\prime}(x)| \le 1$ with equality holding only at $x=0$ (and its $2\pi$-periodic
images). By a direct calculation, we have for $|x|<\eta$,
\begin{align*}
(\partial_x w_0)^3 - \partial_x w_0  & = (1-5 \eta x^4)^3 - (1-5 \eta x^4) = O(x^4).
\end{align*}
Clearly it holds that
\begin{align*}
\partial_{xx} \Bigl( (\partial_x w_0)^3 - \partial_x w_0  \Bigr) \Bigr|_{x=0}=0.
\end{align*}

Now since
\begin{align*}
\partial_t (w_x) = (w_x^3-w_x)_{xx} - \nu \partial_x^5 w,
\end{align*}
we have
\begin{align*}
(\partial_t \partial_x w)(0,0)  = ( (\partial_x w_0)^3- \partial_x w_0)_{xx} \Bigr|_{x=0} -\nu \partial_x^5 w_0
\Bigr|_{x=0} \notag  =120 \nu \eta>0.
\end{align*}

Since $A(t)=(\partial_x w)(t,0)$ is a continuously differentiable function of $t$ with $A(0)=1$, $A^{\prime}(0)>0$,
obviously \eqref{pf_thm0_1} holds.

Step 2. The perturbation argument.

Let $\phi \in C_c^{\infty}(\{x:\, |x|  <\eta  \})$ be a fixed smooth cut-off function with
$\phi(x)=1$ for $|x| <\frac {\eta}2$.  Let $\phi$ be even in $x$ and let
\begin{align*}
v_0^{\delta}(x)= w_0(x) - \delta x \phi(x).
\end{align*}
Note that $v_0^{\delta}$ is odd in $x$ and still has mean zero.

Clearly
\begin{align} \label{step2_1}
\| v_0^{\delta} -w_0\|_{H^2} \le \delta \| x \phi(x) \|_{H^2} \le \operatorname{const} \cdot \delta
\end{align}
and can be made arbitrarily small.

On the other hand for $|x|<\eta/2$,
\begin{align*}
\partial_x v_0^{\delta}(x) = \partial_x w_0(x) - \delta =1-5 \eta x^4 -\delta \le 1-\delta.
\end{align*}
For $\eta/2 \le |x| \le \pi$, since by construction we have
\begin{align*}
|\partial_x w_0(x) | \le 1-\beta,
\end{align*}
for some constant $\beta>0$. Obviously by choosing $\delta>0$ sufficiently small we can have
\begin{align*}
| \partial_x v_0^{\delta}(x) | \le 1 -\frac {\beta}2, \quad \forall\; \eta/2\le |x| \le \pi.
\end{align*}
Therefore we have shown
\begin{align*}
\|  \partial_x v_0^{\delta} \|_{\infty} <1.
\end{align*}

Now let $v^{\delta}$ be the solution to \eqref{2} corresponding to initial data $v_0^{\delta}$.
By \eqref{step2_1}, \eqref{pf_thm0_1} and Proposition \ref{prop_stable}, for $\delta>0$ sufficiently small, we have
\begin{align*}
\| \partial_x v^{\delta}(t^*)\|_{\infty}> C_1^{\prime}>1,
\end{align*}
where $C_1^{\prime}$ is another constant.

Define $\mathcal A=\{v_0^{\delta}: \; \text{$\delta$ is sufficiently small} \}$.
This concludes our construction.
\qquad\end{proof}

\begin{proof}[Proof of Corollary \ref{cor0}]
The essential ideas are already in the proof of Theorem \ref{thm0}. Therefore we only
sketch the necessary notational modifications.

Take $\eta>0$ sufficiently small and $a= \frac 1 {\sqrt d} (1,\cdots, 1)^T$
(here $d$ is the dimension). Note that by definition $|a|=1$. We  define a
smooth function $w_0 \in C^{\infty}(\mathbb T^d)$ such that
\begin{align*}
w_0(x) = a \cdot x - \eta \sum_{j=1}^d x_j^5, \quad \text{for $|x|<\eta$}.
\end{align*}
Let $D=[-\pi,\pi]^d$ be the fundamental domain of the torus $\mathbb T^d$.
For $|x|\ge \eta$, $x\in D$, we simply require
\begin{align*}
|\nabla w_0(x)|<1.
\end{align*}

Take a radial $\phi \in C_c^{\infty}(\{ x \in \mathbb R^d: \; |x|<\eta\})$ such that
$\phi(x) \equiv 1$ for $|x| \le \eta/2$.

For $\delta>0$ sufficiently small, define
\begin{align*}
v_0^{\delta} x = w_0(x) - \delta\cdot (a\cdot x)\cdot \phi(x)
\end{align*}
and
\begin{align*}
\mathcal A = \{ v_0^{\delta}:\; \text{$\delta>0$ is sufficiently small}\}.
\end{align*}

The set $\mathcal A$ is the desired family of initial data.
\qquad\end{proof}

\section{Proof of Theorem \ref{thm1}}
In this section we give the proof of Theorem \ref{thm1}.

\begin{proof}[Proof of Theorem \ref{thm1}]
Without loss of generality we assume the dimension $d=1$. The case
$d\ge 2$ can be proved with suitable modifications.

Fix $\epsilon>0$. Let
\begin{align*}
f(x) = \frac 1 {2\pi} \int_{\mathbb R} e^{-\xi^4} e^{i \xi \cdot x}
d\xi.
\end{align*}
Define
\begin{align*}
 C_1 = \| f \|_{L_x^1(\mathbb R)}, \quad
 A_1 = \| f^{\prime\prime} \|_{L_x^1(\mathbb R)}.
\end{align*}
Define $t_1>0$ such that
\begin{align} \label{pf_thm1_step1_e0}
2C_1^3 \cdot A_1 \cdot \nu^{-\frac 12} \cdot 2 t_1^{\frac 12} =
\frac {\epsilon} 3.
\end{align}

\texttt{Step 1}: We show that there exist $t_2>0$ with $t_2\le t_1$
and $h_0 \in C^{\infty}(\mathbb T)$ with mean zero such that $\|
\partial_x h_0 \|_{\infty} <1$ and
\begin{align} \label{pf_thm1_step1_e1}
\| e^{-\nu t_2 \partial_{xxxx}} \partial_x h_0 \|_{\infty} > C_1 -
\frac {\epsilon} 3.
\end{align}

To show this, we first choose $\tilde F (t,x)$ to be an odd function
of $x$ which is $2\pi$-periodic, and  such that
\begin{align*}
\tilde F(t,x)=
\begin{cases}
\int_0^x \operatorname{sgn}(f(s/(\nu t)^{\frac 14} ) ) ds, \quad
0\le x \le t^{\frac 15};\\
0, \qquad t^{\frac 15} + | \int_0^{t^{\frac 15}}
\operatorname{sgn}(
f(s/(\nu t)^{\frac 14} ) ) ds| \le x \le \pi;\\
\text{linear interpolation}, \quad t^{\frac 15} \le x \le t^{\frac
15} + | \int_0^{t^{\frac 15}}  \operatorname{sgn}( f(s/(\nu
t)^{\frac 14} ) ) ds|.
\end{cases}
\end{align*}

Easy to check that for $t\le 1/2$ the function $\tilde F(t,x)$ is
well-defined. Furthermore
\begin{align*}
\partial_x \tilde F(t,x) = \operatorname{sgn}(f(x/(\nu t)^{\frac 14}
)), \qquad a.e.\; |x|\le t^{\frac 15};
\end{align*}
and $\| \partial_x \tilde F\|_{\infty} \le 1$. Define
\begin{align*}
\tilde G(t,x) = \Bigl( e^{-\nu t \partial_{xxxx}} ( \partial_x
\tilde F(t,\cdot) ) \Bigr)(t,x).
\end{align*}

Then clearly if $t$ is sufficiently small, then
\begin{align*}
|\tilde G(t,0)| & \ge \int_{|x|\le t^{\frac 15}} |f(\frac x {(\nu
t)^{\frac 14}} ) | (\nu t)^{-\frac 14} dx - \int_{|x|> t^{\frac 15}}
|f(\frac x {(\nu t)^{\frac 14}} ) |  (\nu t)^{-\frac 14} dx
\notag \\
& = \|f \|_{L_x^1(\mathbb R)} - 2 \int_{|x|> t^{\frac 15}} |f(\frac
x {(\nu t)^{\frac 14}} ) |  (\nu t)^{-\frac 14} dx \notag \\
& = C_1 - 2 \int_{|x|>\nu^{-\frac 14}
t^{-\frac 1{20}}} |f(x)| dx \notag \\
& > C_1- \frac{\epsilon} 4.
\end{align*}
In the last inequality above, we used the fact that $f$ is a
Schwartz function and the tail contribution to the integral can be
made arbitrarily small (by taking $t$ small).

Now take an even function $\psi \in C_c^{\infty}(\mathbb R)$ such that $0\le \psi\le 1$, $\psi(x)=1$
for $|x| \le 1$ and $\int \psi =1$. Define $\psi_{\delta}(x) = \delta^{-1}
\psi(x/\delta)$ and
\begin{align*}
\tilde F_{\delta}(t,x)= (1-\delta) \cdot \Bigl(\psi_{\delta}\ast
\tilde F(t,\cdot)\Bigr)(t,x),
\end{align*}
where $\ast$ is the usual convolution on $\mathbb R$. Easy to check
that $\| \partial_x \tilde F_{\delta} \|_{\infty}<1$, $\tilde
F_{\delta}$ is $2\pi$-periodic, odd in $x$ and has mean zero.

Define
\begin{align*}
\tilde G_{\delta}(t,x) = \Bigl( e^{-\nu t \partial_{xxxx}} (
\partial_x \tilde F_{\delta}(t,\cdot) ) \Bigr)(t,x).
\end{align*}

Obviously for $\delta$ sufficiently small, we have
\begin{align*}
|\tilde G_{\delta}(t,0)|>C_1- \frac{\epsilon} 3.
\end{align*}
Thus \eqref{pf_thm1_step1_e1} is achieved with $h_0(x) =\tilde
F_{\delta}(t,x)$.

\texttt{Step 2}: Control of the nonlinear solution. We shall fix
$t_2$ and $h_0$ from Step 1. With $h_0$ as initial data, let $h$ be
the corresponding solution to \eqref{2}. We argue by contradiction
and assume that
\begin{align} \label{pf_thm1_step1_e3}
\sup_{0\le t \le t_2} \| \partial_x h(t,\cdot)\|_{\infty} \le
C_1-\epsilon.
\end{align}
Then
\begin{align*}
\| h_x^3 - h_x \|_{\infty} \le 2C_1^3, \qquad \forall\, 0<t\le t_2.
\end{align*}

Now since
\begin{align*}
\partial_x h (t) = e^{-\nu t \partial_{xxxx}} \partial_x h_0 +
\int_0^t \partial_{xx} e^{-\nu s \partial_{xxxx}} \Bigl(
(h_x^3-h_x)(t-s) \Bigr) ds,\notag \\
\end{align*}
we get
\begin{align*}
  \| \partial_x h(t) - e^{-\nu t \partial_{xxxx}} \partial_x
 h_0\|_{\infty} \notag
 \le \int_0^t \| \partial_{xx} e^{-\nu s \partial_{xxxx}} (
 (h_x^3- h_x)(t-s) ) \|_{\infty} ds.
 \end{align*}

Regard $(h_x^3-h_x)$ as a $2\pi$-periodic function on $\mathbb R$.
Recall that $f^{\prime\prime}(x) = \mathcal F^{-1}( -\xi^2
e^{-\xi^4})$. Then
\begin{align*}
  & \| \partial_{xx} e^{-\nu s \partial_{xxxx}} ( (h_x^3-h_x) )
  \|_{L_x^{\infty}(\mathbb T)} \notag \\
= & \| \partial_{xx} e^{-\nu s \partial_{xxxx}} ( (h_x^3-h_x) )
  \|_{L_x^{\infty}(\mathbb R)} \notag \\
  \le & \| \mathcal F^{-1} ( - |\xi|^2 e^{-\nu s |\xi|^4} )
  \|_{L_x^1(\mathbb R)} \| h_x^3 - h_x \|_{L_x^{\infty}(\mathbb R)}
  \notag \\
  \le &  \| f^{\prime\prime} \|_{L_x^1(\mathbb R)} \cdot (\nu s)^{-\frac 12}  \cdot 2
  C_1^3 \notag \\
  = & A_1 \cdot (\nu s)^{-\frac 12} \cdot 2 C_1^3.
  \end{align*}

Thus we obtain for $0<t\le t_2$,
\begin{align*}
   \| \partial_x h(t) - e^{-\nu t \partial_{xxxx}} \partial_x
 h_0\|_{\infty} \le A_1 \cdot 2 \nu^{-\frac 12} t_2^{\frac 12} \cdot 2C_1^3.
 \end{align*}

Since $t_2\le t_1$, by \eqref{pf_thm1_step1_e0} and Step 1, we get
\begin{align*}
  \| \partial_x h(t_2) \|_{\infty} > C_1 -\frac {\epsilon} 3 -
  \frac{\epsilon} 3= C_1-\frac{2\epsilon} 3
  \end{align*}
  which is an obvious contradiction to
  \eqref{pf_thm1_step1_e3}.
\qquad\end{proof}

\subsection*{Acknowledgments.}
 D. Li was supported in part by an Nserc discovery grant.
 The research of Z. Qiao is partially supported by Hong Kong Research
Council GRF grant PolyU 15302214, and NSFC/RGC Joint Research Scheme N\underline{\;\;}HKBU204/12.
 The research of T. Tang is mainly supported by Hong Kong Research Council GRF Grants
 and Hong Kong Baptist University FRG grants.

\frenchspacing
\bibliographystyle{siam}

\end{document}